\documentclass[psamsfonts,12pt]{amsart}

\usepackage{amssymb,amsfonts}
\usepackage{subcaption}
\usepackage{fancyhdr}
\usepackage{amsfonts}
\usepackage{amsthm}
\usepackage{amsmath}
\usepackage{amssymb}
\usepackage{latexsym}
\usepackage{setspace}
\usepackage{multicol}
\usepackage{float}
\usepackage{graphicx}
\usepackage{multirow}
\usepackage{enumerate}
\usepackage[font=footnotesize, labelfont=sf]{caption}
\usepackage{verbatim}
\usepackage[pdftex,pagebackref,hypertexnames=false, colorlinks, citecolor=black, linkcolor=blue, urlcolor=blue]{hyperref}
\usepackage{color}
\usepackage{wrapfig}
\usepackage{tikz}
\usepackage{pgflibraryplotmarks}
\usepackage{xspace}
\usepackage{graphicx}
\usepackage{tabularx}
\usepackage{algorithm}
\usepackage{bbold}
\usepackage{mathtools}
\usepackage[left=1in,right=1in,top=1in,bottom=1in]{geometry}

\usetikzlibrary{arrows, positioning,chains,fit,shapes,calc,matrix,scopes,decorations.pathmorphing,intersections}

\newcommand{\RR}{\mathbb{R}}

\newcommand{\PP}{\mathbb{P}}
\newcommand{\TT}{\mathbb{T}}
\newcommand{\PT}{\mathbb{P}\mathbb{T}}

\newcommand{\mcH}{\mathcal{H}}

\newcommand{\mcS}{\mathcal{S}}

\newcommand{\la}{\lambda}
\allowdisplaybreaks

\DeclareMathOperator{\Span}{span}
\DeclareMathOperator{\conv}{conv}
\DeclareMathOperator{\pos}{pos}
\DeclareMathOperator{\tconv}{tconv}

\DeclareMathOperator{\Supp}{Supp}
\DeclareMathOperator{\tropDet}{tropDet}

\makeatletter

\newtheorem{thm}{Theorem}[section]
\newtheorem*{thm*}{Theorem}
\newtheorem{lem}[thm]{Lemma}

\newtheorem{prop}[thm]{Proposition}

\newtheorem{cor}[thm]{Corollary}
\theoremstyle{definition}

\newtheorem{example}[thm]{Example}
\newtheorem{rmk}[thm]{Remark}

\definecolor{cqcqcq}{rgb}{0.7529411764705882,0.7529411764705882,0.7529411764705882}
\definecolor{ffffff}{rgb}{1.,1.,1.}
\definecolor{cqcqcq}{rgb}{0.7529411764705882,0.7529411764705882,0.7529411764705882}
\definecolor{qqwuqq}{rgb}{0.,0.39215686274509803,0.}
\definecolor{qqttcc}{rgb}{0.,0.2,0.8}
\definecolor{ccqqqq}{rgb}{0.8,0.,0.}

\newenvironment{exmp}
  {\pushQED{\qed}\example}
  {\popQED\endexample}

\title{Tropical convex hullS of polyhedral sets}

\author{Cvetelina Hill}
\address{
Georgia Institute of Technology\\
  North Ave NW\\
  Atlanta, GA\\
  USA 30332
   }
\email{cvetelina.hill@gatech.edu}

\author{Sara Lamboglia}
\address{
Institut f\"ur Mathematik\\
 Goethe-Universit\"at Frankfurt\\
 Robert-Mayer-Str. 6-8\\
  60325 \indent Frankfurt a. M.\\
   Germany
   }
  \email{lamboglia@math.uni-frankfurt.de}

\author{Faye Pasley Simon}
\address{North Carolina State University\vspace{-.75em}}
\address{
Greensboro College\\
815 West Market Street\\
Greensboro\\
NC, USA 27401
   }
\email{faye.simon@greensboro.edu}

\begin{document}

\maketitle


\begin{abstract}
    In this paper we focus on the tropical convex hull of convex sets and polyhedral complexes. We give a vertex description of the tropical convex hull of a line segment and a ray. 
    Next we show that tropical convex hull and ordinary convex hull commute in two dimensions and characterize tropically convex polyhedra in any dimension. 
    Finally we show that the dimension of a tropically convex fan depends on the coordinates of its rays and give a lower bound on the degree of a fan tropical curve using only tropical techniques. 
\end{abstract}

\section*{Introduction}
\label{sec:intro}

Tropical convexity is the analog of classical convexity in the {\em tropical semiring} $(\RR,\oplus,\odot)$ where $a\oplus b=\min(a,b),$ and $a\odot b = a+b.$ The goal of this paper is to explore the interplay between tropical convexity and its classical counterpart. Our aim is to describe  the tropical convex hull of polyhedra, polyhedral complexes, and in particular, tropical curves.
\par The primary focus of tropical convexity is the study of \textit{tropical polytopes}: the tropical convex hull of finite sets. These are  widely studied \cite{D-S,C-G-Q-S,C-G-Q,G-S,J,G-M,A-G-G-2} and find applications in various areas of mathematics. Recently, techniques from tropical convexity have been applied to mechanism design \cite{C-T}, optimization \cite{A-G-G-3}, and maximum likelihood estimation \cite{R-S-T-U}. Some specific applications are  the resolution of monomial ideals \cite{D-YU}, and discrete event dynamic systems \cite{B-C-G-Q}. Moreover, computational tools exist to aid in further study of tropical polytopes \cite{J-2,A-G-G-2}.
\par A tropical polytope is not always classically convex, but does have an explicit description as the finite union of some ordinary polytopes \cite{D-S}.
Tropical polytopes which are also ordinary polytopes are called \textit{polytropes} as discussed in \cite{J-K}. However, there exist ordinary polytopes which are tropically convex, but are not finitely generated (for an example, see Figure \ref{fig:Notdistributive}). Here we further examine this relationship between classical convexity and tropical convexity by  studying the structure of the tropical convex hull of polyhedral sets. Our first result is the following:
\begin{thm*}[Theorems \ref{thm: PequalsT} and \ref{thm: tconv commute in R2}]
If $a,b \in \RR^n$ and $U \subset \RR^2$, then 
\begin{enumerate}
    \item[(i)] $\tconv\conv(a,b) = \conv\tconv(a,b)$;
    \item[(ii)] $\tconv \pos (a) = \pos \tconv(0,a)$;
    \item[(iii)] $\tconv\conv U = \conv\tconv U$.
\end{enumerate}
\end{thm*}
Ordinary and tropical convex hull do not commute as in part (i) even for small examples (e.g., triangles) in dimension $3$. However, the  tropical convex hull of an ordinary polyhedron is itself an ordinary polyhedron. We characterize which ordinary polyhedra are tropically~convex.
\begin{thm*}[Theorem~\ref{thm: tropconv polythedra}]
A full-dimensional ordinary polyhedron is tropically convex if and only if all of its defining halfspaces are tropically convex.
\end{thm*}

Many properties and theorems valid in classical convexity are also valid in the tropical setting; for example, separation of convex sets~\cite{C-G-Q,G-S}, Minkowski-Weyl Theorem~\cite{G-K,G-K-2,J},  Carath\'eodory and Helly  Theorems~\cite{D-S,G-M}, and Farkas Lemma~\cite{D-S}.

Here we consider the  classical result in algebraic geometry  (see for example \cite{E-H}) which  bounds the degree of a projective variety $X$ from below by
 \begin{equation}\label{eq: class}  \dim\Span X -\dim X +1 \leq \deg X. \end{equation}
Our first aforementioned result describing the tropical convex hull of line segments and rays provides some information on the dimension of tropical convex hulls. Using this result we study a tropical analogue of (\ref{eq: class}) in the case of tropical curves. We can substitute  $\Span X$ either with the tropical convex hull of  a tropical curve $\Gamma$ or with a tropical linear space of smallest dimension containing $\Gamma$. The latter may not be unique and it is not easy to determine. Thus, we choose to replace $\Span X$ with $\tconv\Gamma.$ The tropical analogue of (\ref{eq: class}) we consider is 
\begin{equation}
    \label{eq: tropClass}
     \dim \tconv\Gamma \le \deg \Gamma.
\end{equation}
If $\Gamma$ is realizable, then this follows immediately from the classical inequality (\ref{eq: class}). In Section~\ref{sec: tropC-2} we give  a proof of (\ref{eq: tropClass}) for fan tropical curves that  relies entirely on tropical techniques. 
\par The structure of this paper is as follows. In Section~\ref{sec:onedim} we recall basic definitions of tropical convexity. Then
we describe the tropical convex hull of a line segment and a ray as ordinary polyhedra. Using this result we show the dimensions are easily calculable using coordinates of the respective endpoints. We also prove that ordinary and tropical convex hull commute in two dimensions. In Section~\ref{sec:higher-dim} we prove that convexity and polyhedrality are preserved after taking the tropical convex hull. Next we classify tropically convex ordinary halfspaces, linear spaces, and polyhedra. Finally, in Section~\ref{sec: tropC-2}, we use our results to prove the inequality (\ref{eq: tropClass}) in the case of fan tropical curves. 

\section{
Line segments, rays, and sets in $\RR^2$}
\label{sec:onedim}

Key definitions from tropical convexity are presented in the first part of this section. A description of the tropical convex hull of any arbitrary set is given in Proposition~\ref{prop: tconv intersection}. In Theorem $\ref{thm: PequalsT}$ we show that ordinary and tropical convex hull commute in any dimension in the case of two points. This allows us to find the dimension of the tropical convex hull of a line segment or a ray using the coordinates of its endpoints in Corollary~\ref{cor: dim of line segment}. We also prove ordinary and tropical convex hull always commutes in two dimensions in Theorem~\ref{thm: tconv commute in R2}.

A set $U\subset\RR^n$ is \textit{tropically convex} if $(a\odot x)\oplus(b\odot y)$ is in $U$ for any $x,y\in U$ and $a,b\in\RR$ with $a\oplus b=0$.  The \textit{tropical convex hull} of $U\subset \mathbb R^{n}$ is the smallest tropically convex set that contains $U.$ This is defined equivalently in \cite{G-K} by 
\begin{equation}\label{def: tconvInf}
    \tconv U=\bigcup_{V\subset U : |V|<\infty }\tconv V.
\end{equation}

If $V=\{v_1,\ldots,v_k\}$ is a finite set, then by \cite[Definition 2.1]{G-K} its tropical convex hull is 
given by
\begin{equation*}\label{def:tconv}
    \tconv V=\left\{ a_1\odot v_1\oplus \cdots \oplus a_k\odot v_k \mid a_i \in \RR,\hspace{1em} \bigoplus_{i=1}^k a_i=0  \right\}.
\end{equation*}
\noindent Furthermore, points in $\tconv V$ can be characterized by types as defined in \cite{D-S}. Let $[n]=\{1,\ldots,n\}$ and $[n]_0 = \{0,1,\ldots,n\}$. Given a point $x\in \mathbb R^n,$
the \textit{type of $x$ relative to $V$}, or \textit{covector} in \cite{F-R,L-S}, is the $n$-tuple $T_x = \left( T_1, \ldots, T_n\right)$ such that $T_j \subseteq [k]$ for all $j$, and 
    $i \in T_j $ if $\min(v_i-x)$ is obtained in the $j$th coordinate.
This is equivalent to saying that $i\in T_j$ if $x\in v_i+\mathcal S_j,$ where $\mcS_j$ is a \textit{sector} of $\mathbb R^n$ spanned by $\{-e_i:i\in[n]\}$ for $j=0$, and $\{e_0, -e_i:i\in[n], i\neq j\}$ for $j\in[n]$.
Here $e_1,\ldots,e_n$ represent the standard unit vectors in $\RR^n$ with $e_{ij}=1$ if $i=j$ and $e_{ij}=0$ otherwise. We denote the vector $\sum_{i=1}^ne_i$ by $e_0$.
The cone $\mcS_j$ is the closure of one of the $n+1$ connected components of $\mathbb R^n\setminus L_{n-1}.$ Here $L_{n-1}$ denotes the max-standard tropical hyperplane, or the tropicalization of $V(x_1+\ldots+x_n+1)$ with the max convention, whose cones are  $\pos(-e_{i_1},\ldots,-e_{i_n})$. 
\par The proof of the Tropical Farkas Lemma  \cite{D-S} states that $x\in \tconv V$ if and only if the $j$th entry of  $T_x$ is nonempty for all $j$, meaning there exists at least one $v_i$ such that  $x\in v_i+\mcS_j$ \cite[Lemma 28]{J-L}.
As a consequence, we have the following proposition which also holds true in the case of $U\subset(\RR\cup\{\infty\})^{n}$  \cite[Proposition 7.3]{L-S}. We give here a proof for completeness. Figure \ref{fig:example TP^2} gives an example of (\ref{eq: intersection of minkowski sums}) in $\mathbb R^2.$

\begin{prop}\label{prop: tconv intersection}
If $U \subset \mathbb R^n$, then the tropical convex hull of $U$ is equal to the intersection of the Minkowski sums of $U$ with each of the sectors. That is
 \begin{equation}
   \label{eq: intersection of minkowski sums}
    \tconv U=\bigcap_{j=0}^n(U+\mcS_j).
    \end{equation}
\end{prop}

\begin{proof}
If $x\in \tconv U$, then (\ref{def: tconvInf})  implies that $x\in \tconv V$ for some finite set $V\subset U.$ By the Tropical Farkas Lemma~\cite{D-S}  we obtain $x\in \bigcap_{j=0}^n(V+\mcS_j)$, hence $x\in \bigcap_{j=0}^n(U+\mcS_j).$
 On the other hand, if $x\in  \bigcap_{j=0}^n(U+\mcS_j),$ then there exist $u_1,\ldots, u_n \in U$ such that  $x\in u_j+\mcS_j$ for every $j.$ For $V=\{u_1,\ldots, u_n\}$ it follows that $x\in \bigcap_{j=0}^n(V+\mcS_j)= \tconv V \subset \tconv U.$
\end{proof}


\begin{figure}
\centering
    \begin{tikzpicture}[scale=0.16, line cap=round,line join=round,x=1.0cm,y=1.0cm]
\clip(-10.1,-8.5) rectangle (9.94,6.42);
\draw [->,line width=1.pt] (0.,0.)-- (-8.,0.);
\draw [->,line width=1.pt] (0.,0.)-- (0.,-8.);
\draw [->,line width=1.pt] (0.,0.)-- (5.5,5.5);
\node[label={$\mcS_0$}] at (-5,-8) {};
\node[label={$\mcS_1$}] at (4,-4) {};
\node[label={$\mcS_2$}] at (-2,1.5) {};
\end{tikzpicture}
\hspace{1em}
\begin{tikzpicture}[scale=0.35]
    \filldraw[fill=gray, fill opacity=0.4] (0,2) -- (1,0) -- (3,3) -- (1,5) -- cycle;
    \node[circle, fill=black, minimum size=2mm,inner sep=0, outer sep=0] (a) at (0,2) {};
    \node[circle, fill=black, minimum size=2mm,inner sep=0, outer sep=0] (b) at (1,0) {};
    \node[circle, fill=black, minimum size=2mm,inner sep=0, outer sep=0] (c) at (3,3) {};
    \node[circle, fill=black, minimum size=2mm,inner sep=0, outer sep=0] (d) at (1,5) {};
    \end{tikzpicture}   
    \hspace{1em}
\begin{tikzpicture}[scale=0.35]
    \filldraw[draw=none, fill=gray,fill opacity=0.4] (-2,5) -- (1,5) -- (3,3) -- (3,-0.5) to[bend left] (-1,0) to[bend right] cycle;
    \node[circle, fill=black, minimum size=2mm,inner sep=0, outer sep=0] (a) at (0,2) {};
    \node[circle, fill=black, minimum size=2mm,inner sep=0, outer sep=0] (b) at (1,0) {};
    \node[circle, fill=black, minimum size=2mm,inner sep=0, outer sep=0] (c) at (3,3) {};
    \node[circle, fill=black, minimum size=2mm,inner sep=0, outer sep=0] (d) at (1,5) {};
    \draw[dashed] (0,2) -- (1,0) -- (3,3) -- (1,5) -- cycle;
    \draw (-2,5)--(1,5)
          (3,3)--(3,-0.5)
          (1,5)--(3,3);
    \end{tikzpicture}
    \hspace{1em}
    \begin{tikzpicture}[scale=0.35]
    \filldraw[draw=none, fill=gray,fill opacity=0.4] (1,5) -- (0,2) -- (0,-1) to[bend left] (2.5,-0.5) to[bend right] (3,7) -- cycle;
    \node[circle, fill=black, minimum size=2mm,inner sep=0, outer sep=0] (a) at (0,2) {};
    \node[circle, fill=black, minimum size=2mm,inner sep=0, outer sep=0] (b) at (1,0) {};
    \node[circle, fill=black, minimum size=2mm,inner sep=0, outer sep=0] (c) at (3,3) {};
    \node[circle, fill=black, minimum size=2mm,inner sep=0, outer sep=0] (d) at (1,5) {};
    \draw[dashed] (0,2) -- (1,0) -- (3,3) -- (1,5) -- cycle;
    \draw (1,5)--(3,7)
            (0,2)--(0,-1)
            (0,2)--(1,5);
     
    \end{tikzpicture}
    \hspace{1em}
    \begin{tikzpicture}[scale=0.35]
    \filldraw[draw=none, fill=gray,fill opacity=0.4] (1,0) -- (5,4) to[bend right] (1,5.5) to[bend right] (-2,0) -- cycle;
    \node[circle, fill=black, minimum size=2mm,inner sep=0, outer sep=0] (a) at (0,2) {};
    \node[circle, fill=black, minimum size=2mm,inner sep=0, outer sep=0] (b) at (1,0) {};
    \node[circle, fill=black, minimum size=2mm,inner sep=0, outer sep=0] (c) at (3,3) {};
    \node[circle, fill=black, minimum size=2mm,inner sep=0, outer sep=0] (d) at (1,5) {};
    \draw[dashed] (0,2) -- (1,0) -- (3,3) -- (1,5) -- cycle;
    \draw (-2,0)--(1,0)
            (1,0)--(5,4);
    \vspace{2em}
    \end{tikzpicture}
    \hspace{1em}
    \begin{tikzpicture}[scale=0.35]
    \filldraw[fill=gray,fill opacity=0.4] (0,0) -- (1,0) -- (3,2) -- (3,3) -- (1,5) -- (0,2) -- cycle;  
    \node[circle, fill=black, minimum size=2mm,inner sep=0, outer sep=0] (a) at (0,2) {};
    \node[circle, fill=black, minimum size=2mm,inner sep=0, outer sep=0] (b) at (1,0) {};
    \node[circle, fill=black, minimum size=2mm,inner sep=0, outer sep=0] (c) at (3,3) {};
    \node[circle, fill=black, minimum size=2mm,inner sep=0, outer sep=0] (d) at (1,5) {};
    \draw[dashed] (0,2) -- (1,0) -- (3,3) -- (1,5) -- cycle;
    \end{tikzpicture}

    \caption{Illustration of Proposition \ref{prop: tconv intersection} in $\PT^2$. From left to right: The three sectors, a polytope $P$, the Minkowski sums $P+\mcS_0, P+\mcS_1,$ $P+\mcS_2,$ and $\tconv P$.}
    \label{fig:example TP^2}
\end{figure}
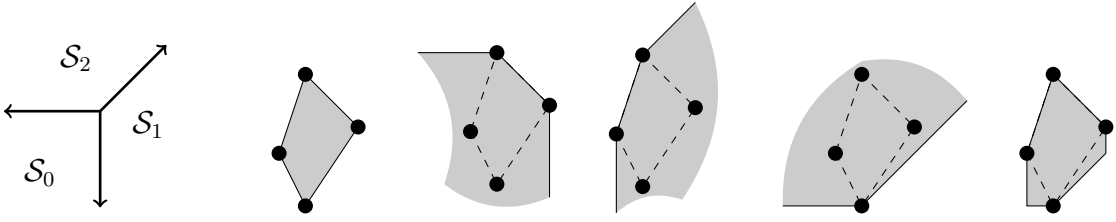

As a direct consequence of Proposition~\ref{prop: tconv intersection} we obtain Corollary~\ref{cor: tconvPisconvex}. Note that it can  also be proven directly by using the definition of tropical convex hull.
Lemma~\ref{cor: convtconvVintconvP} shows that repeatedly taking the convex hull and tropical convex hull of a
set stabilizes after one step.
\begin{cor}\label{cor: tconvPisconvex}
    If $P \subset \RR^n$ is convex, then $\tconv P$ is convex.
\end{cor}

\begin{cor}
\label{cor: convtconvVintconvP}
If $U \subset \mathbb R^n$, then
$
\label{eq: corollary of tconvPisConvex}
    \tconv\conv U=\tconv(\conv \tconv U).
$
\end{cor} 

\begin{proof}
The forward direction is immediate since $\conv U \subset \conv\tconv U.$
The containment $\tconv U \subseteq \tconv\conv U$ and Corollary~\ref{cor: tconvPisconvex} imply $\conv\tconv U \subseteq \tconv\conv U,$
so its tropical convex hull is also contained in $\tconv\conv U$.
\qedhere
\end{proof}

Let $a$ and $b$ be points in  $\mathbb R^n.$ For the remainder of the section we assume that 
    \begin{equation}\label{eq: a and b assumptions}
        a=(0,\ldots,0) \text{ and } 0 < b_1 < \cdots < b_n.
    \end{equation}
In this case, using \cite[Proposition 4]{D-S}, the tropical line segment  $\tconv(a,b)$ is a concatenation of line segments with $n+1$ \textit{pseudovertices} in $\mathbb R^n$ given by $p_0 = a$ and
\begin{equation}\label{eq: pseudovertices}
    p_j = (b_1,\ldots,b_{j-1},b_j,\ldots,b_j) \text{ for } j \in [n].
\end{equation}

If $a$ and $b$ do not satisfy (\ref{eq: a and b assumptions}), we can apply first a  linear transformation which translates $a$ to the origin and  then another that  relabels coordinates  so that  $0  \le b_1 \le \ldots \le b_n.$ 
If $b_i = b_j$ for some $i \neq j$ or $b_j=0$ for some $j$, then the pseudovertices of $\tconv(a,b)$ lie in the tropically convex hyperplane $x_i-x_j=0$ or $x_j=0$ and the same holds for $\conv\tconv(a,b)$ \cite[Theorem 2]{D-S}. Thus $\tconv\conv(a,b)$ and $\conv\tconv(a,b)$ lie in the hyperplane $x_i-x_j=0$ or $x_j=0$. Each of these hyperplanes  is isomorphic to $\mathbb R^{n-1}.$ We can repeat this process until the appropriate projection of $b$ has distinct positive coordinates.

The following theorem shows that 
the tropical convex hull and convex hull commute for two points in $\RR^n$ for all $n$.

\begin{thm}
\label{thm: PequalsT}
If  $a,b$ are points in $\mathbb R^n,$ then 
\begin{itemize}
\item[(i)] $\tconv \conv(a,b) = \conv \tconv(a,b);$

\item[(ii)] 
$\tconv \pos(a)=\pos \tconv(0,a).$ \end{itemize}
\end{thm}
Corollary~\ref{cor: tconvPisconvex} implies the forward containment of Theorem~\ref{thm: PequalsT}(i).
For the converse, we use an explicit description of $\conv\tconv(a,b)$ given in the following lemma. 

\begin{lem}
\label{lem: H-rep} If $a, b \in \mathbb R^n$ satisfy $a = (0,\ldots,0) \text{ and } 0 <b_1<\cdots<b_n,$
then $\conv \tconv(a,b)$ is a full-dimensional simplex whose $\mcH-$representation is given by
\begin{equation}
    \begin{split}
      \label{eq:HrepOfQ}
      b_1 - x_1 & \geq 0 \\
      -(b_{j+1}-b_j)x_{j-1}
      +(b_{j+1}-b_{j-1})x_j-(b_j-b_{j-1})x_{j+1} & \geq 0 \text{\hspace{1em} for $j \in [n-1].$}\\
      -x_{n-1}+x_{n} & \geq 0
    \end{split}
  \end{equation}
\end{lem}

\begin{proof}
Observe that the vertices of $\conv \tconv(a,b)$ are the pseudovertices 
$p_0,\ldots,p_{n}$ of $\tconv(a,b)$ 
as described in (\ref{eq: pseudovertices}). 
These are $n+1$ affinely independent points of $\RR^n$ since the vectors $p_1-a=p_1,\ldots,p_{n-1}-a=p_{n-1},b-a=b$ are linearly independent. This implies  $\conv \tconv(a,b)$ is a simplex. Hence, each of its $n+1$ facets is the convex hull of $n$ vertices.  
To show that (\ref{eq:HrepOfQ}) is the $\mcH-$representation of $\conv \tconv(a,b)$ we will show that the corresponding equation of each one  of the $n+1$ inequalities is the hyperplane supporting one of the facets of $\conv \tconv(a,b).$

Let $x=(x_1,\ldots,x_n)$ be a point in $ \conv \tconv(a,b) =  \conv(a,p_1,\allowbreak \ldots,p_{n-1},b).$  The $j$th coordinate of $x$ is given by
$$
    x_j = \lambda_1b_1+\ldots+\lambda_{j-1}b_{j-1}+\left(\lambda_j+\lambda_{j+1}+\ldots +\lambda_n\right)b_j
$$
where $\lambda_1+\ldots+\lambda_n\le 1$ and $\lambda_i\ge 0$ for every $i$.
Substituting the coordinates of $x$ into the first linear form of (\ref{eq:HrepOfQ}) we obtain
$(1-\lambda_1-\dots-\lambda_n)b_1$. Since   $\lambda_1+\ldots+\lambda_n\le 1$ and $b_1\ge 0$ it follows that $b_1-x_1\ge 0$.
Note that  equality occurs if and only if 
$x$ is in the facet $ \conv(p_1,\ldots,p_{n-1},b)$. Thus, $b_1-x_1=0$ defines this facet of  $\conv \tconv(a,b)$,
that is $\{b_1-x_1=0\}\cap \conv\tconv(a,b)=\conv(p_1,\ldots,p_{n-1},b)$.

After substituting into the second linear form of (\ref{eq:HrepOfQ}) we have that 
$$
    -(b_{j+1}-b_j)x_{j-1}+(b_{j+1}-b_{j-1})x_j-(b_j-b_{j-1})x_{j+1}  = \lambda_j(b_{j-1}-b_j)(b_{j}-b_{j+1}).
$$
\noindent Since $\lambda_j\ge 0$ and  $b_j \ge b_{j-1}$ for each $j$, we know $x$ satisfies the second inequality. Here equality occurs if and only if $x$ is in the facet $ \conv(a,p_1,\ldots,\allowbreak p_{j-1}, p_{j+1}, \ldots, p_{n-1},b),$ so  
\[-(b_{j+1}-b_j)x_{j-1}+(b_{j+1}-b_{j-1})x_j-(b_j-b_{j-1})x_{j+1} = 0\] defines this facet of $\conv \tconv\allowbreak(a,b)$ for each $j\in[n-1].$

Lastly, we have that $-x_{n-1}+x_n = \lambda_n(b_n-b_{n-1}) \ge 0$.
Equality holds if and only if $x$ is in the facet $\conv(a, p_1,\ldots,p_{n-1})$, and hence this facet is defined by $-x_{n-1}+x_n=0$.
\end{proof}

\begin{lem}
\label{lem: VinP} 
If $a, b \in \mathbb R^n$ and $V$ is a finite subset of $\conv(a,b)$, then \[\tconv(V)\subset\conv \tconv(a,b).\]
\end{lem}

\begin{proof}
Without loss of generality, assume $a = (0,\ldots,0)$ and $0<b_1<\ldots<b_n.$ Let $V~=~\{\lambda_1b,\lambda_2b,\ldots,\lambda_rb\} \subset \conv(a,b)$ for some parameters $\lambda_i \in [0,1]$. Assume the parameters are ordered $0 \leq \lambda_1 \leq \lambda_2 \leq\ldots\leq \lambda_r \leq 1$. Take $x \in \tconv V$ and let $T_x$ be the type of $x$ relative to $V.$ 
By \cite[Lemma~10]{D-S}, 
the point $x$ satisfies 
    \begin{equation}\label{eq: xinXS}
       x_k-x_j \leq \lambda_i(b_k-b_j)\text{ for $j,k \in [n]$ with $i \in T_j$.}
    \end{equation}  
    We will show that $x$ satisfies the $\mcH$-representation of $\conv \tconv(a,b)$ given in Lemma~\ref{lem: H-rep}.  Since 
    the union of all coordinates $T_j$ of $T_x$ covers $[r],$
    (\ref{eq: xinXS}) implies that 
$$ 0 \leq \dfrac{x_{j+1}-x_j}{b_{j+1}-b_j} \leq \dfrac{x_{j}-x_{j-1}}{b_{j}-b_{j-1}} \leq 1 \text{\hspace{1em} for all $j \in [n-1].$}$$ 
For $j=1$, this implies $\dfrac{x_{1}}{b_{1}} \leq 1$, so $b_1-x_1 \geq 0$. 
For $j \in [n-1]$, rewriting the inequality $\dfrac{x_{j+1}-x_j}{b_{j+1}-b_j} \leq \dfrac{x_{j}-x_{j-1}}{b_{j}-b_{j-1}}$ shows that $-(b_{j+1}-b_j)x_{j-1}     +(b_{j+1}-b_{j-1})x_j-(b_j-b_{j-1})x_{j+1} \geq 0.$
Lastly, if $j = n-1$, then $0 \leq \dfrac{x_n-x_{n-1}}{b_n-b_{n-1}}$, so $-x_{n-1}+x_n \geq 0.$
\end{proof}   

\begin{proof}[Proof of Theorem~\ref{thm: PequalsT}]
For part (i), assume without loss of generality that
$a=(0,\ldots,0)$ and $0<b_1<\cdots<b_n$. 
Corollary~\ref{cor: tconvPisconvex} and the containment $\tconv(a,b) \subset \tconv \conv(a,b)$ imply that $\conv \tconv(a,b) \subseteq \tconv\conv(a,b)$. Now take $x \in \tconv \conv(a,b)$. Since the tropical convex hull of a set is the union of the tropical convex hulls of all its subsets, it follows that
there is a finite set $V \subset \conv(a,b)$ such that $ x \in \tconv(V).$ Lemma~\ref{lem: VinP} implies $\tconv(V) \subset \conv\tconv(a,b)$, so $x\in \conv\tconv(a,b)$. 

To show part (ii), take $x \in \tconv\pos(a).$ There exist scalars $\lambda_0,\ldots,\lambda_n \ge 0$ such that $\lambda_j a \in \pos(a)$ for each $j \in [n]_0$ and $x \in \tconv(0, \lambda_0 a,\allowbreak \ldots, \lambda_n a)$. Assume the scalars are ordered $\lambda_0\le \lambda_1\le \lambda_2 \le \ldots \le \lambda_n$ so $x \in 
\tconv \conv(0, \lambda_n a).$
 By Theorem~\ref{thm: PequalsT}$(i)$ it follows that $x~\in~\conv\tconv(0,\lambda_n a)$. Furthermore, this means $x \in \pos\tconv(0,\lambda_n a).$ The pseudovertices of $\tconv(0,\lambda_n a)$ and $\tconv(0,a)$ are scalar multiples of one another meaning $x \in \pos\tconv(0, a).$  The other inclusion
 $\pos \tconv(0,a)\subset \tconv\pos(0,a)$ follows from Corollary~\ref{cor: tconvPisconvex}. 
\end{proof}

\begin{cor}\label{cor: dim of line segment}
If  $a$ and $b$ are points in $\mathbb R^n$, then 
\begin{itemize}
\item[(i)]  $\dim \tconv\conv(a,b)$ is  the number of nonzero distinct coordinates of $a-b$;
\item[(ii)]  $\dim \tconv\pos(a)$ is  the number of nonzero distinct coordinates of $a$.
\end{itemize}
\end{cor}

\begin{proof}
Part (i) follows from  the proof of Lemma~\ref{lem: H-rep} since $\tconv\conv(a,b)$ is a full-dimensional simplex in $\RR^d$ where $d$ is  the number of nonzero distinct coordinates in $a-b$. 
For part (ii) observe that the generators of $\pos\tconv(0,a)$  are the pseudovertices of $\tconv(0,a)$ which are vertices of $\tconv \conv(0,a)$.
\end{proof}

As a consequence of Corollary~\ref{cor: dim of line segment}  we have the following result for tropically convex fans. An application of this lemma appears in Section~\ref{sec: tropC-2}.

\begin{lem} \label{lem: fan-diffCoord}
If $F$ is a tropically convex fan in $\mathbb R^{n}$, then $\dim F$ is equal to the maximum number of nonzero distinct coordinates of a point in $F.$
\end{lem}

\begin{proof}  
Let $d$ be the maximum number of nonzero distinct coordinates of any point in $F$, and let $x$ be one such point in $F.$
Since $F$ is a tropically convex fan it contains $\tconv\pos(x).$ Corollary~\ref{cor: dim of line segment} implies  that $\dim\tconv \pos(x)= d$, hence $\dim \ F \geq d.$
Suppose that $\dim F > d$. Let $C$ be a cone contained in $F$ such that $\dim C = \dim F$. By hypothesis, each point in $C$ has at most $d$ nonzero distinct coordinates. This implies that $C$ is contained in the union of finitely many linear spaces in $\mathbb R^n$ of dimension at most $d.$ This contradicts the  assumption that $\dim C=\dim F >d.$
\end{proof}

Now we consider arbitrary sets in $\mathbb R^2$ 
and give a generalization of
Theorem~\ref{thm: PequalsT}.

\begin{lem}
\label{lemma: convTconv in 2D}
    If $V \subset \mathbb R^2$ is finite, then
    $
    \tconv \conv V = \conv \tconv V.
    $
\end{lem}

\begin{proof}
We prove the lemma by showing that each vertex  of $  \tconv \conv V $ is either a point in $V$ or a pseudovertex of $\tconv V$. 

By Proposition~\ref{prop: tconv intersection} we know $\tconv \conv V = \bigcap_{j=0}^2(\mcS_j+\conv V)$.
A face of a Minkowski sum of polyhedra is a Minkowski sum of a face from each summand. Since $\mcS_j$ has only one vertex, namely the origin, it follows that the vertices of $\mcS_j+\conv V$ are precisely the vertices of $\conv V$. The facets of $\mcS_j+\conv V$  arise as either the sum of the vertex of $\mcS_j$ and an edge of $\conv V$, or as the sum of a vertex of $\conv V$ and a ray of $\mcS_j$. In the former case, these are simply the edges of $\conv V$. In the latter case, these are the unbounded edges parallel to a ray of $\mcS_j$ and the vertex of each of them is   a vertex $v\in V$. 

From this description of the facets and vertices of $\mcS_j+\conv V$ we deduce that a vertex of $\tconv \conv V$ is either a vertex of $\conv V$ or it is the intersection of a facet of $\mcS_i+\conv V$ and a facet of $\mcS_j+\conv V$ for some $i,j\in[2]_0$.  
Suppose that neither of the facets is an edge of $\conv V$ (Otherwise we would get a vertex of $\conv V$.), then the intersection point is a pseudovertex of $\tconv(v,w)$ and a vertex of $\conv \tconv V$. Suppose that only one of the facets is an edge of $\conv V$. This intersection point must be a vertex of $\conv V$. Otherwise it is in the interior of the edge of $\conv V$, which implies that the ray intersecting the edge also intersects the interior of $\conv V$ and hence is not a facet.  
\end{proof}

\begin{thm}\label{thm: tconv commute in R2}
If $U\subset \RR^2$, then $\tconv \conv U = \conv \tconv U$. 
\end{thm}
\begin{proof}
The forward containment is implied by the fact that $\tconv \conv U$ is convex by Corollary~\ref{cor: tconvPisconvex}.

For backward containment, suppose that $x\in\tconv\conv U$. Then by (\ref{def: tconvInf}) it follows that there exists a finite set $V\subset \conv U$, such that $x\in\tconv V$. The classical Carath\'eodory Theorem implies that each point $v_i \in V$ can be written as a convex combination of finitely many points in $U$. Call this set $A_i\subset U$. Since $V$ is finite, it follows that $A=\bigcup_iA_i$ is a finite subset of $U$ and $V\subset\conv A$. Now we have $x\in\tconv V\subset\tconv\conv A$. It follows $x \in \conv \tconv A$ by Lemma~\ref{lemma: convTconv in 2D}. Since $A \subset U$, this implies $x \in \conv\tconv U$.
\end{proof}

Theorem~\ref{thm: tconv commute in R2} does not hold in general when $n\ge3.$ It is not difficult to find examples for which $\conv(\tconv V)$ is not tropically convex.

\begin{figure}
    \centering
    \includegraphics[scale=0.2]{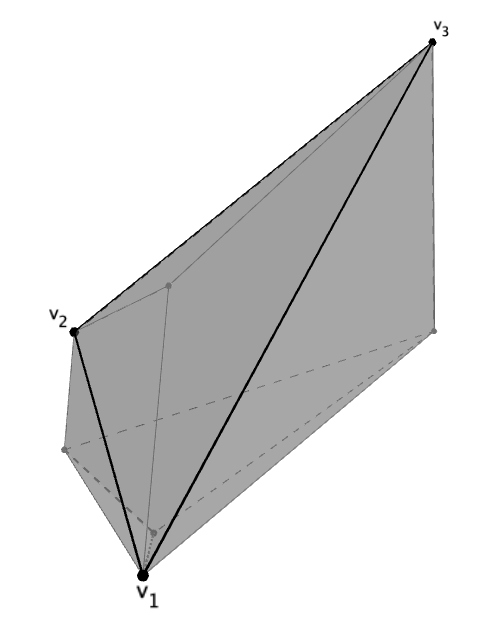}
    \hspace{5em}
    \includegraphics[scale=0.2]{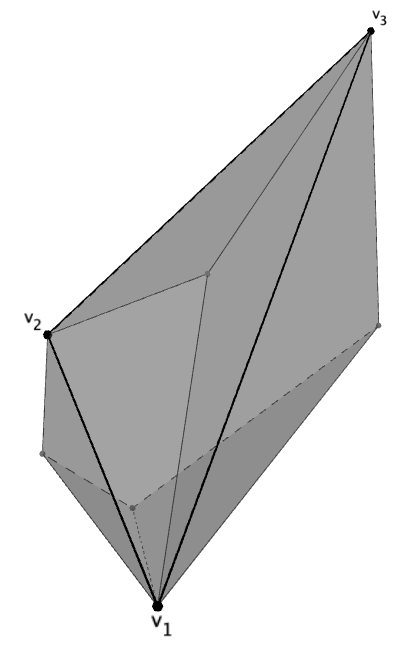}
    \caption{Illustration of Example \ref{exmp: conv of tconv not trop convex}. Left: Convex hull of $\tconv(v_1,v_2,v_3)$ with $P$ in bold. Right: Tropical convex hull of $P$ with $P$ in bold. The polytope on the left is strictly contained in the polytope on the right. }
    \label{fig:tconvNot=conv}
\end{figure}

\begin{exmp}
\label{exmp: conv of tconv not trop convex}
Let $P \subset \mathbb R^3$ be the triangle in Figure \ref{fig:tconvNot=conv} 
with vertices $v_1=(0,0,0),$ $v_2 = (1,2,3),$ and $v_3=(4,1,7)$. The convex hull of $\tconv(v_1,v_2,v_3)$ has $7$ vertices and is not tropically convex. In fact, it is possible to find a point $x$ in the classical line segment $v_1v_3$ such that the tropical convex hull of $x$ and the midpoint of the line segment $v_2v_3$ is not contained in $\tconv(v_1,v_2,v_3)$. 
Using Proposition~\ref{prop: tconv intersection} we compute 
the tropical convex hull of $P$ which is a polytope with $7$ vertices strictly containing $\conv(\tconv(v_1,v_2,v_3))$.
\end{exmp}

\section{
Polyhedral sets}
\label{sec:higher-dim}

In this section we examine the tropical convex hull of arbitrary polyhedral sets, halfspaces, and linear spaces. The main result of this section is Theorem~\ref{thm: tropconv polythedra} which classifies all tropically convex ordinary polyhedra in $\RR^n$. 

\begin{lem}\label{Cor: properties}
If $P\subset\mathbb R^n$ is a polyhedron (resp.\ cone, polyhedral complex, fan, polytope), then $\tconv P$ is a polyhedron (resp.\ cone, polyhedral complex, fan, polytope).
\end{lem}

 \begin{proof}
 If  $P$ is a polyhedron then 
 $\tconv P$ is a polyhedron since it is the intersection of the finitely many polyhedra $P+\mcS_j$.
 If $P$ is a cone then  $P+\mcS_j$ is a cone for every $j$ and (\ref{eq: intersection of minkowski sums}) implies that  $\tconv P$ is also a cone.

Now let $P$ be a polyhedral complex, so $P=\cup_{i=1}^{N} P_i$ where each $P_i$ is a polyhedron. By (\ref{eq: intersection of minkowski sums})  it follows that
$$\tconv P=\tconv \left(\bigcup_{i=1}^N P_i\right)=\bigcap_{j=0}^n\bigcup_{i=1}^N(P_i+\mcS_j).$$
Observe that by distributing the intersection over the union of Minkowski sums we obtain the union of $N^{n+1}$ sets. Each set in the union is an intersection of $n+1$ Minkowski sums of the form 
$
(P_{i_0}+\mcS_0)\cap\ldots\cap(P_{i_{n}}+\mcS_n),
$
where $(i_0,\ldots,i_n)\in\{N\}^{n+1}$, so
\begin{equation*}
\tconv P= \bigcup_{(i_0,\ldots, i_n)\in \{N\}^{n+1}} ((P_{i_0}+\mcS_0)\cap \cdots \cap (P_{i_n}+\mcS_n)  ).
\end{equation*}
It follows that $\tconv P$ is a polyhedral complex since the finite intersection of polyhedra is a polyhedron. In fact, the polyhedral structure may be given by a refinement of the polyhedral complex whose polyhedra are  $\left\{(P_{i_0}+\mcS_0)\cap \cdots \cap (P_{i_n}+\mcS_n)  \right\}_{(i_0,\ldots, i_n)\in \{N\}^{n+1}}.$
If $P$ is a fan, the results on polyhedral complexes and cones imply $\tconv P$ is also a fan. 

Lastly, let $P$ be a polytope. To show $\tconv P$ is a  polytope it suffices to show it is bounded.
Suppose $\tconv P$ is not bounded. Hence it contains a ray $w+\pos(v).$ Since $P$ is bounded, again (\ref{eq: intersection of minkowski sums}) implies that $\pos(v)$ is contained in each sector $\mcS_j.$ This is not possible since the intersection of all sectors is the origin.
\end{proof}

Using the following lemma, we classify all tropically convex ordinary halfpaces in Proposition~\ref{prop:tconvHalfspace}.

\begin{lem}\label{lem:halfspaces and sectors}
Let $\mathcal H$ be a halfspace in $\mathbb R^n$. If $\mcS_j$ is one of the standard sectors in $\RR^n$ for $j\in[n]_0$, then either $\mathcal H+\mathcal S_j=\mathcal H$ or $\mathcal H+\mathcal S_j=\mathbb R^n$.
\end{lem}

\begin{proof} Let $\mcH$ be defined by $\left\{\sum_{k=1}^na_kx_k\ge 0\right\}$ and let $\mcS_j$ be one of the standard sectors in $\RR^n$ for $j\in[n]_0$. If $\mcS_j\subset\mcH$, then it follows immediately that $\mcH+\mcS_j=\mcH$.

Suppose that $\mcS_j\not\subset\mcH$. 
This means that at least one of the rays $\pos e_i, i\ne j,$ generating $\mcS_j$ is contained in $\mcH^c$; equivalently  $\sum_{k=1}^na_ke_{ik}<0$. We will consider two cases. First, suppose that $i=0$, and recall that $e_0=(1,\ldots, 1)$. It follows that $\pos e_0\not\subset\mcH$, and hence $\sum_{k=1}^na_k<0$. If $y\in\mcH^c$, then  $\sum_{k=1}^na_ky_k<0$. Let $\lambda\in\RR$ such that
\[
\lambda\ge\frac{\sum_{k=1}^na_ky_k}{\sum_{k=1}^na_k}>0, 
\] 
which implies $ \lambda\sum_{k=1}^na_k  \le \sum_{k=1}^na_ky_k$. Hence, $ 0  \le \sum_{k=1}^na_k(y_k-\lambda)$, implying that for any $y\in\mcH^c$, the point $y-\lambda e_0\in\mcH$ for the choice of $\lambda$ specified above.

For the second case let $e_i$ be the vector containing a $-1$ in position $i$ and $0$ otherwise. Suppose that $\pos  e_i\not\subset\mcH$ and let $y\in\mcH^c$. Then we have that $\sum_{k=1}^na_ke_{ik}=-a_i<0$ and $\sum_{k=1}^na_ky_k<0$. Let $\lambda\in\RR$ be such that 
\[
\lambda\ge-\frac{\sum_{k=1}^na_ky_k}{a_i}>0.
\] 
Hence, $-\lambda\sum_{k=1}^na_ke_{ik}+\sum_{k=1}^na_ky_k  \ge 0$ and   $\sum_{k=1}^na_k(y-\lambda e_{ik})  \ge 0$. It follows that $y-\lambda e_i\in\mcH$. 

This shows that if $\mcS_j\not\subset\mcH$, then any point in $\mcH^c$ can be written as $(y-\lambda e_i)+\lambda e_i, i\ne j$, with $y-\lambda e_i\in \mcH$. Hence, $\mcH+\mcS_j=\RR^n$.
\end{proof}

\begin{prop}\label{prop:tconvHalfspace}
If $\mathcal H$ is a halfspace in $\mathbb R^n$, then either $\tconv \mathcal H =\mathcal H$ or $\tconv \mathcal H=\mathbb R^n$.
\end{prop}
\begin{proof}
By Proposition~\ref{prop: tconv intersection} we know  $\tconv \mcH = \bigcap_{j=0}^n(\mcS_j+\mcH)$. Using Lemma~\ref{lem:halfspaces and sectors}, if there exists $j\in[n]_0$ such that $\mcS_j\subset\mcH$, then $\tconv\mcH = \mcH$. Otherwise $\tconv\mcH=\RR^n$.
\end{proof}

The following proposition shows that a halfspace is tropically convex if and only if either all of the entries of its inner normal vector are nonpositive, or it contains at most one positive entry such that the sum of all entries is nonegative.

\begin{prop}\label{prop:classifyHalfspaces}
A halfspace $\mcH=\{\sum_{k=1}^{n}a_kx_k\ge0\}$ in $\RR^{n}$ is tropically convex if and only if there exists a $j\in[n]_0$ such that $S_j\subset \mathcal H$. This happens if and only if exactly one of the following conditions is satisfied.
\begin{itemize}
    \item[(i)] If $a_k\le 0$ for every $k\in[n]$, then $\mathcal S_0\subset\mcH$.
    \item[(ii)] If $a_j\ge 0$, $a_k\le0$ for every $k\neq j$, and $a_j+\sum a_k\ge 0$, then $\mathcal S_j\subset\mcH$.
\end{itemize}
\end{prop}

\begin{proof}
The first statement follows immediately from Proposition~\ref{prop:tconvHalfspace}. The sector $\mcS_j$ is contained in $\mcH$ if and only if the spanning rays $e_i$ of $\mcS_j$ for $i \neq j$ satisfy the inequality $\sum_{k=1}^na_ke_{ik}\ge 0$. This inequality is satisfied precisely in cases (i) and (ii) listed above. 
\end{proof}

\begin{lem}\label{lem:hyper}
A linear space is tropically convex if and only if it is an intersection of hyperplanes of the form 
$\{x_i-x_j=0 \mid i \neq j\}$ or $\{x_k=0\}$. 
\end{lem}

\begin{proof}
By \cite[Theorem 2]{D-S}, hyperplanes of the form  $\{x_i-x_j=0\}$  and  $\{x_i=0\}$ are tropically convex. Hence, the intersection of any hyperplanes of this form is also tropically convex.

Conversely, let $L \subset \RR^n$ be a linear space and suppose $L$ is tropically convex. Consider $\conv(0, x)$ for some $x\in L$. By Corollary~\ref{cor: dim of line segment}, the dimension of the tropical convex hull of $\conv(0,x)$ is equal to the number of distinct nonzero coordinates of $x$. Since $L$ is tropically convex, $x$ has at most $\dim L$ distinct nonzero coordinates.
This implies $L$ is contained in the union of the intersections of some hyperplanes $\{x_i-x_j=0\}$  and  $\{x_k=0\}$. Since $L$ is convex, it follows that $L$ is just an intersection of  $\{x_i-x_j=0\}$  and  $\{x_k=0\}$  for some $i\neq j$ and $k$. 
\end{proof}

The following theorem is the main result of this section.
\begin{thm}\label{thm: tropconv polythedra}
A full-dimensional ordinary polyhedron is tropically convex if and only if all of its defining halfspaces are tropically convex. 
\end{thm}
\begin{proof} Let $P \subset \RR^n$ be a full-dimensional, ordinary polyhedron. Since $P$ is full-dimensional, it has a unique, irredundant hyperplane representation. If all defining halfspaces of $P$ are tropically convex, then $P$ is tropically convex.

Suppose that $P$ is tropically convex and there exists a defining halfspace $\mcH$ of $P$ that is not tropically convex. Let $H$ be the hyperplane at the boundary of $\mcH$. Since $\mcH$ is not tropically convex, it follows that $H$ is not tropically convex. Otherwise, by Lemma~\ref{lem:hyper} $H$ is parallel to one of the facets of the standard tropical hyperplane, so both $\mcH$ and $-\mcH$ are tropically convex. Let $x', y' \in \mcH$ such that $\tconv(x',y') \not\subset \mcH$. This implies that there exist $x,y \in \tconv(x',y')\cap H$ such that $(\tconv(x,y)\setminus \{x,y\})\subset \mathcal H^c$.
Hence, at least one pseudovertex $p$ of $\tconv(x,y)$ is in $\mcH^c$. 
 After relabeling, we assume the coordinates of $y-x$ are ordered $$y_1-x_1\le\cdots\le y_s-x_s\le 0\le y_{s+1}-x_{s+1}\le\cdots\le y_n-x_n.$$
Generalizing the result \cite[Proposition 3]{D-S}  there are two forms for the pseudovertices of $\tconv(x,y)$ in $\RR^n$ based on the signs of the coordinates of the difference $y-x$.
For any $s<j\le n$ the pseudovertex is $$p=(y_1, y_2, \ldots, y_j, y_j-x_j+x_{j+1}, \ldots, y_j-x_j+x_n)$$ and for $j\le s$ the pseudovertex is $(y_1-y_j+x_j, \ldots, y_{j-1}-y_j+x_j, x_j, x_{j+1}, \ldots, x_n)$. We provide the computation for the former and omit it for the latter as the proof is analogous. Since $p\in\mcH^c$, it follows that $\sum_{k=1}^n a_kp_k<0$.

Using a translation $T$ along $H$ we can translate $x$ and $y$ so that at least one of the points $Tx$ or $Ty$ is contained in $P$. Without loss of generality, we may assume that $Tx\in P$. If $Ty\in P$, then we are done. Suppose that $Ty\not\in P$. Consider the line segment $\conv(Tx,Ty)\subset H$, which must intersect the boundary of $P$. Let the point of intersection be $z$, which can be written as $z=\lambda Tx + (1-\lambda) Ty$, for $0<\lambda<1$. We claim that $\tconv(Tx,z)\not\subset\mcH$.

Note that one of the pseudovertices of $\tconv(Tx,z)$ is $p'=(z_1,z_2,\ldots, z_j, z_j-Tx_j+Tx_{j+1}, \ldots, z_j-Tx_j+Tx_n)$. We will show that $p'\not\in\mcH$. Note that 
\begin{align*}
    \sum_{k=1}^na_kp_k & = a_1y_1+\cdots+a_jy_j+a_{j+1}(y_j-x_j+x_{j+1})+\cdots+a_n(y_j-x_j+x_n)\\
    & = \sum_{k=1}^ja_ky_k+\sum_{k=j+1}^na_k(y_j-x_j)+\sum_{k=j+1}^na_kx_k<0
\end{align*}
where the inequality is preserved under the translation $T$. That is, $\sum_{k=1}^na_kTp_k<0$.

We compute the following:
\begin{align*}
    \sum_{k=1}^na_kp'_k & = a_1z_1+a_2z_2+\cdots+a_jz_j+a_{j+1}(z_j-Tx_j+Tx_{j+1})+\cdots+a_n(z_j-Tx_j+Tx_n)\\
    & = \sum_{k=1}^ja_kz_k+\sum_{k=j+1}^na_k(z_j-Tx_j)+\sum_{k=j+1}^na_kTx_k\\
    & = \sum_{k=1}^ja_k(\lambda Tx_k+(1-\lambda)Ty_k)+\sum_{k=j+1}^na_k(\lambda Tx_j+(1-\lambda)Ty_j-Tx_j)+\sum_{k=j+1}^na_kTx_k\\
    & = (1-\lambda)\sum_{k=1}^ja_kTy_k+(1-\lambda)\sum_{k=j+1}^na_k(Ty_j-Tx_j)+(1-\lambda)\sum_{k=j+1}^na_kTx_k.
\end{align*}
Hence, $\sum_{k=1}^na_kp'_k=(1-\lambda)\sum_{k=1}^na_kTp_k<0$.
This implies that there are two points in $P$, namely $Tx$ and $z$, whose tropical convex hull is not in $P$. This contradicts the assumption that $P$ is tropically convex.
\end{proof}

\begin{cor}
Let $P\subset\RR^n$ be a 
polyhedron of dimension $d<n$. $P$ is tropically convex if and only if it is contained in a tropically convex linear space $L$ of dimension $d$  and its $\mcH$-representation  in $L$ is given only by tropically convex halfspaces.
\end{cor}

\begin{proof}
After translation, we may assume that $P$ contains the origin.
Hence, $P$ is contained in a unique, $d$-dimensional linear subspace $L$. 
If $L$ is tropically convex, then by Lemma~\ref{lem:hyper} $P$ is contained in the intersection of finitely many hyperplanes of the form $\{x_k=0\}$ for $k\in [n]$, and $\{x_i-x_j=0 \mid i\ne j\}$ for $i,j\in[n]$. Now we can work in $L$ by deleting the $x_k$ and $x_i$ coordinates. Note that the restriction of this projection map to $P$ is an isomorphism. 
We now consider $P$ in the $d$-dimensional linear subspace $L$. Equivalently, we can work in $\RR^d$ where $P$ is full-dimensional and has a unique, irredundant halfspace representation. 
By Theorem~\ref{thm: tropconv polythedra} it follows that $P$ is tropically convex in $L$ if and only if the halfspaces defining $P$ in $L$ are tropically convex. Hence, the inner normal vectors of the defining halfspaces satisfy Proposition~\ref{prop:classifyHalfspaces}. 
The lift of each of these hyperplanes to $\RR^n$ will have the same equation, hence it still satisfies the conditions of Proposition~\ref{prop:classifyHalfspaces}. Therefore, each halfspace in $L$ is tropically convex in $L$ if and only if it is tropically convex in $\RR^n$. 

Suppose that $L$ is not tropically convex. Then there exist two points $x,y\in L$ such that $\tconv(x,y)\not\subset L$. Using a translation argument similar to that in the proof of Theorem~\ref{thm: tropconv polythedra}, we can find two points $Tx, z \in P$ whose tropical convex hull is not contained in $P$. 
Hence, $P$ is not tropically convex. 
\end{proof} 

\begin{rmk}
The authors of \cite{F-K} characterize  distributive polyhedra. Any such polyhedron $P$ has the property   that $\min(x,y)$ and $\max(x,y)$ are contained in $P$.  Note that only polytropes are distributive polytopes. This is not true for tropically convex polyhedra. 
For example, consider the triangle $P\subset\mathbb R^2$  in Figure \ref{fig:Notdistributive} whose vertices are the origin, $(3,1)$, and $(1,3)$. This is a tropically convex polytope by Theorem \ref{thm: tconv commute in R2}, but not a distributive polytope. In particular, it is not max-closed since $\max(B,C)\notin P$.
\end{rmk}
\definecolor{wqwqwq}{rgb}{0.3764705882352941,0.3764705882352941,0.3764705882352941}
\definecolor{ududff}{rgb}{0.30196078431372547,0.30196078431372547,1.}
\definecolor{uuuuuu}{rgb}{0.26666666666666666,0.26666666666666666,0.26666666666666666}
\begin{figure}
    \centering
\begin{tikzpicture}[line cap=round,line join=round,>=triangle 45,x=0.5cm,y=0.5cm]
\draw[->,color=black] (-1,0.) -- (5,0.);

\draw[->,color=black] (0.,-1) -- (0.,5);

\draw[color=black] (0pt,-10pt) node[right] {\footnotesize $0$};
\clip(-2,-2) rectangle (8,6.16);
\fill[line width=0.8pt,color=wqwqwq,fill=wqwqwq,fill opacity=0.10000000149011612] (0.,0.) -- (3.,1.) -- (1.,3.) -- cycle;
\draw [line width=0.8pt,color=wqwqwq] (0.,0.)-- (3.,1.);
\draw [line width=0.8pt,color=wqwqwq] (3.,1.)-- (1.,3.);
\draw [line width=0.8pt,color=wqwqwq] (1.,3.)-- (0.,0.);
\draw [line width=1.pt] (0.,0.)-- (1.,1.);
\draw [line width=1.pt] (1.,3.)-- (1.,1.);
\draw [line width=1.pt] (1.,1.)-- (3.,1.);
\begin{scriptsize}

\draw (3.3,1.11) node {$B$};
\draw (1.5,1.9) node {$P$};
\draw (0.84,3.37) node {$C$};
\draw  (3.,3.) circle (0.5pt);
\draw(4.7,3.17) node {$\max(B,C)$};
\end{scriptsize}
\end{tikzpicture}
    \caption{A tropically convex triangle that is not distributive since the point $\max(B,C)$ is not contained in it. In black the tropical convex hull of the vertices.}
    \label{fig:Notdistributive}
\end{figure}
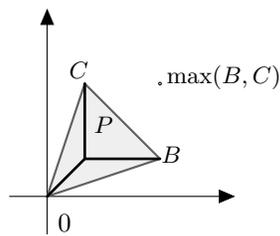

\section{Lower bound on the degree of a tropical curve}
\label{sec: tropC-2}


For the remainder of the paper we use an alternative definition of the tropical convex hull from \cite[Proposition 4]{D-S}. We work in the tropical projective torus $\PP\TT^n \cong \RR^{n+1}/\RR\textbf{1}$ which is isomorphic to $\RR^n$ as follows. Given a set $U \subset \RR^{n+1}$, its tropical convex hull is the set of all possible tropical linear combinations
$a_1\odot~u_1\oplus\ldots\oplus~a_k\odot~u_k$ with $u_i \in U$ and $a_i \in \RR $. With this definition we have $\tconv U+\RR\textbf{1}=\tconv U$. Taking the quotient with $\RR\textbf{1}$ we obtain \vspace{-5mm}\\
 \[
    \tconv U = \tconv\left\{ u \in \RR^{n}:  (0,u)+\RR\textbf{1}\subset U\right\}
 \]
computed using (\ref{def: tconvInf}). It follows  that the results obtained in Section~\ref{sec:onedim} also hold in this case. 

Let $\Gamma$ be a tropical  curve. This is a weighted balanced rational polyhedral complex of dimension one  in $\PT^{n}$. 
The degree of $\Gamma$ is defined to be the multiplicity at the origin of the stable intersection  between $\Gamma $ and the standard tropical hyperplane \cite[Definition 3.6.5]{M-S}.
For realizable curves, this is equal to the  degree of any classical curve which tropicalizes to $\Gamma$ \cite[Corollary 3.6.16]{M-S}.
Let $r_1,\ldots,r_k$ be the rays of a tropical curve $\Gamma$ where  $r_i=w+\pos(v_i) $ for some $w\in\PT^n$.
Since $\Gamma\subset \PT^n$ we can choose each $v_i\in\PT^n$ to be the minimal nonnegative integer vector representative that generates $r_i.$
If the multiplicity of the ray $r_i$ in $\Gamma$ is $m_i,$ then by \cite[Lemma 2.9]{B-G-S} we have
\begin{equation}\label{degree}
    (\deg\Gamma)\mathbf{1}=\sum_{i=1}^k m_iv_i.
\end{equation} 

\noindent The main result of this section is Theorem~\ref{thm:degree general}, which states that a tropical ffan curve $\Gamma$ satisfies the inequality
\begin{equation*} \tag{\ref{eq: tropClass}}
 \dim\tconv \Gamma\le \deg \Gamma.
\end{equation*}
The proof relies entirely on tropical and combinatorial techniques and uses results from Sections \ref{sec:onedim} and \ref{sec:higher-dim}. Here we state the following two results we reference within the subsequent proofs.

\begin{thm}\label{thm:tropRank}\cite[Theorem 4.2]{DSS} 
The tropical rank of a $k\times n$ matrix $M$ is equal to one plus the  dimension of the tropical convex hull of the columns of $M$ in $\mathbb R^{k}/\mathbb R\mathbf 1$.
\end{thm}

\begin{lem}\label{lem:tropSingular}\cite[Lemma 5.1]{RGST}
An $n\times n$ matrix $M$ is singular if and only if its rows lie on a tropical hyperplane in $\mathbb R^n/\mathbb R\mathbf 1$.
\end{lem}
 
As a first step towards proving  (\ref{eq: tropClass}), we prove the following lemma.

\begin{lem}\label{lem: degree finite}
If $\Gamma \subset \PT^n$ is a fan tropical curve and $W \subset \Gamma$ is finite, then $$\dim\tconv W \leq \deg \Gamma.$$
\end{lem}

\begin{proof}
Let $\deg\Gamma = d$ and $\Gamma$ be given by rays $r_1=\pos(v_1),\ldots,r_k=\pos(v_k)$ with minimal nonnegative vectors $v_1,\ldots,v_k$. Let $W \subset \Gamma$ be a finite set of points and $\Supp W$ denote the set of minimal nonnegative vectors of rays which contain a point of $W$. That is, $$\Supp W = \{ v_i \mid w \in \pos(v_i) \text{ for some } w \in W\}.$$
\par First suppose $|\Supp W| = 1$, so $W \subset r_i$ for some $i \in [k]$ and $\dim\tconv W \leq \dim\tconv r_i$. Each ray of $\Gamma$ has at most $d$ nonzero distinct entries since $\deg \Gamma = d$. By Lemma~\ref{lem: fan-diffCoord} this means $\dim\tconv r_i \leq d$ for all $i \in [k]$ and $\dim\tconv W\leq d$.
\par Let $M$ be the $(n+1)\times k$ matrix whose columns are $v_1,\ldots,v_k$.
We also assume $n+1$, $k \geq d+2$. Otherwise, the result is trivially true. 
We will show that the tropical rank of $M$ is at most $d+1$, implying that $\tconv(v_1,\ldots,v_k) \leq d$. 
Let $D$ be any $(d+2)\times(d+2)$ submatrix of $M$. Each row of $D$ has all nonnegative entries and must have at least two zeros because $\deg \Gamma=d$. Hence, the rows of $D$ lie in 
the tropicalization of the ordinary hyperplane $V(x_0+\ldots+x_{d+1})$ in $\PP\TT^{d+1}$.  By Lemma~\ref{lem:tropSingular} this implies $D$ is tropically singular, so the tropical rank of $M$ is at most $d+1$. Using Theorem~\ref{thm:tropRank} we deduce that  the dimension of the tropical convex hull of the columns of $M$ is at most $d$. 
\par Now suppose $|\Supp W| = |W|$, so each point of $W$ is on a distinct ray of $\Gamma$. More specifically, each point of $W$ is a classical scalar multiple of some distinct $v_i$. The tropical convex hull of any $d+2$ columns of $M$ has dimension at most $d$ and the same holds if each column is scaled since the location of the zero entries is not affected. 
\par Next suppose $1 < |\Supp W| < |W|$ and let $W=\{w_1,\ldots,w_s\}$. Let $M'$ be the $(n+1)\times s$ matrix whose columns are $w_1,\ldots,w_s$. More specifically, its columns are classical scalar multiples of some $v_i$s in $\Supp W$. We know from the previous case that $M$ is tropically singular and the tropical rank is at most $d+1$.  By Lemma~\ref{lem:tropSingular} we have that the columns of any $(d+2)\times(d+2)$ submatrix of $M$ are contained in some hyperplane in $\PP\TT^{d+1}$. If a point is contained in a tropical hyperplane, so is any classical scalar multiple of that point since any tropical hyperplane is a fan. 
For this reason, the columns of any $(d+2)\times(d+2)$ submatrix of $M'$ must also  be contained in at least one of these hyperplanes of $\PP\TT^{d+1}$ from before. Therefore, $M'$ has tropical rank at most $d+1$ and $\dim\tconv W \leq d$.
\end{proof}

\begin{thm}\label{thm:degree general}
If $\Gamma\subset \PT^n$ is a fan tropical curve,
 then $\dim \tconv \Gamma \leq \deg \Gamma$.
\end{thm}

\begin{proof}
Let $\deg \Gamma = d$ and suppose $\dim\tconv\Gamma= d+1$. Since $\tconv \Gamma$ is a fan, there exists a  point 
$p$ with $d+2$ distinct coordinates by Lemma~\ref{lem: fan-diffCoord}. Moreover, $\Gamma$ contains the ray $\pos(p)$. Note that we can choose $p$ to be the minimal nonnegative integer vector that generates this ray. Since $p$ has $d+2$ distinct coordinates, we may assume that $0=p_0<p_1<\cdots<p_{d+1}$. Let $\la_ip$ be $d+2$ distinct points on the ray $\pos(p)$ and assume
 $\la_1<\la_2<\cdots<\la_{d+2}$. Let $M_p$ be the $(n+1)\times (d+2)$ matrix whose columns are $\la_ip$ for $i\in[d+2]$. Then, up to permutation of rows, $M_p$ contains the $(d+2)\times(d+2)$ submatrix 
 \[
 D = \begin{pmatrix}
0 & 0 & \ldots & 0\\
\la_1p_1 & \la_2p_1  & \ldots & \la_{d+2}p_1\\
\vdots & \vdots & \ddots & \vdots\\
\la_1p_{d} & \la_2p_{d} & \ldots & \la_{d+2}p_{d}\\
\la_1p_{d+1} & \la_2p_{d+1} & \ldots & \la_{d+2}p_{d+1}
\end{pmatrix}.
 \]
We will show that $D$ has tropical rank $d+2$ by showing that the tropical determinant of $D$ has a unique minimum attained on its antidiagonal. 
    Using Laplace expansion along the first row, we write the tropical determinant of $D$ as
    $$\tropDet(D) = \min_{i \in [d+2]} 0 + \tropDet(D_i)$$ where $D_i$ is the $(d+1)\times(d+1)$ submatrix of $D$ obtained by deleting its first row and $i$th column.
    We first claim that $\tropDet(D_i) = m_i$ for any $i \in [d+2]$ where $$m_i = \la_1p_{d+1}+\la_2p_d+\la_3p_{d-1}+\cdots+\la_{i-1}p_{d-i+3}+\la_{i+1}p_{d-i+2}+\cdots+\la_{d+1}p_2+\la_{d+2}p_1.$$ Recall that for a $(d+1)\times (d+1)$ matrix $X$, its tropical determinant can be written \[
\tropDet(X) = \bigoplus_{\sigma\in S_{d+1}}x_{1\sigma(1)}\odot x_{2\sigma(2)}\odot\cdots\odot x_{d+1,\sigma(d+1)}.
\]
Let
\begin{align*}
\sigma(m_i) = \la_1p_{\sigma(d+1)}+\la_2p_{\sigma(d)}&+\la_3p_{\sigma(d-1)}+\cdots+\la_{i-1}p_{\sigma(d-i+3)}\\
&+\la_{i+1}p_{\sigma(d-i+2)}+\cdots+\la_{d+1}p_{\sigma(2)}+\la_{d+2}p_{\sigma(1)}.
\end{align*}
Any permutation $\sigma$ can be decomposed into adjacent transpositions of the form $\tau = (j, j+1)$. It suffices to show that $m_i < \tau(m_i)$ to conclude $m_i<\sigma(m_i)$
for any permutation $\sigma\in S_{d+1}$. Let $\tau(m_i)$ represent the expression $m_i$ where $p_{j}$ and $p_{j+1}$ have been exchanged.
First, suppose that $j>d-i+2$, which implies that 
$$m_i-\tau(m_i) = (\lambda_{d-j+2}-\lambda_{d-j+1})(p_{j}-p_{j+1}) < 0.$$
Similarly, if $j< d-i+2$, then  $$m_i-\tau(m_i) = (\lambda_{d-j+3}-\lambda_{d-j+2})(p_{j}-p_{j+1}) < 0.$$ 
If $j=d-i+2$, then $$m_i-\tau(m_i) = (\lambda_{i+1}-\lambda_{i-1})(p_{d-i+2}-p_{d-i+3}) < 0.$$
It follows that $m_i<\tau(m_i)$ for any transposition $\tau=(j,j+1).$

\par Finally, we have $\tropDet(D) = \min_{i\in[d+2]} m_i$. For any $i \in [d+1]$ 
$$m_{i+1} - m_{i} = (a_{i}-a_{i+1})p_{d-i+2} < 0.$$ 
meaning $m_{i+1} < m_i$.  Hence the unique minimum is obtained for $i=d+2$. This implies $D$ has tropical rank at least $d+2$, so by Theorem~\ref{thm:tropRank} the dimension of the tropical convex hull of the columns of $D$ is at least $d+1$ which contradicts Lemma~\ref{lem: degree finite}.
\end{proof}

The following proposition shows that (\ref{eq: tropClass}) holds  for some special types of  tropical curves which are not fans.

\begin{prop}\label{prop:degree}
Let $\Gamma$ be a tropical  curve in $\PT^n$ with rays $r_1,\dots, r_k$. 
If 
$
 \dim \tconv\Gamma=\max_{i\in[k]} \{\dim \tconv r_i\}$,
then 
$
 \dim \tconv \Gamma \le  \deg\Gamma.  
$

\end{prop}
\begin{proof}

 Let $ \dim \tconv\Gamma=\max_{i\in[k]} \{\dim \tconv r_i\}=d$ and $v_1,\dots, v_k\in\PT^{n}$ be the minimal nonnegative integer vectors such that $r_i = w_i+ \pos(v_i) \subset \PT^{n} \text{ for } i\in[k].$
 Then there exists some $j\in[k]$ such that $\dim\tconv~r_j=d$. By 
 Corollary~\ref{cor: dim of line segment} $v_j$ has $d+1$ distinct entries. Hence the maximum component of $v_j$ is at most $d.$ By (\ref{degree}) we have that
 $\dim \tconv\Gamma = d \le  \deg\Gamma$.
\end{proof}

\noindent However, Proposition~\ref{prop:degree} does not hold for all tropical curves.
\begin{exmp}
Let $\Gamma$ be the fan tropical curve in $\PT^2$ with rays spanned by $(0,1,0)$, $(0,0,1)$, $(0,0,-1),$ and $(0,-1,0)$  emanating from the origin.  Each ray $r\subset\Gamma$ is tropically convex so 
$\max_{r\in\Gamma}\{\dim\tconv r\}=1.$
However, 
$\dim\tconv\Gamma=2.$ In fact, $\tconv(\pos(0,-1,0), \allowbreak \pos(0,0,1))$ is the 2-dimensional cone spanned by $(0,-1,0)$ and $(0,0,1)$.
\end{exmp}

Finally, we give an example of a tropical curve where the smallest dimension of a linear space containing it is larger than the dimension of the tropical convex hull of the curve. 

\begin{exmp}
Consider the tropical curve $\Gamma_F$ over the field of Puiseux series $\mathbb C\!\{\!\{t\!\}\!\}$ given by the fan whose rays are the columns of $M_F$ :
\begin{equation*}
\label{eq:MF}
M_F=\begin{pmatrix}
1&1&1&0&0&0&0\\
1&0&0&1&1&0&0\\
1&0&0&0&0&1&1\\
0&1&0&1&0&1&0\\
0&0&1&0&1&1&0\\
0&0&1&1&0&0&1\\
0&1&0&0&1&0&1
\end{pmatrix}.
\end{equation*}
The curve $\Gamma_F$ has degree $3$ and there is no $2$ dimensional tropical linear space containing it \cite[Section 5.3]{M-S}. We now  prove that $\dim \tconv \Gamma_F=2$.

Let $v_1, v_2, \ldots, v_7 \in \mathbb \PT^6$ denote the columns of $M_F$. Using  Macaulay2 \cite{M2} we compute that the tropical rank of $M_F$ is $3$. By Theorem~\ref{thm:tropRank}  $\dim\tconv(v_1,\ldots,v_7)=2$ hence  $\dim\tconv \Gamma_F\ge2.$
We will show that $\dim\tconv V\leq 2$ for any finite $V\subset\Gamma_F.$ Note that this is not implied by Lemma~\ref{lem: degree finite}.

For a  finite set  $V\subset\Gamma_F$  we can consider $\Supp V$ as in the proof of Lemma~\ref{lem: degree finite}.
Suppose that $|\Supp V|=7$, implying that each point of $V\subset\Gamma_F$ is on a distinct ray.  The tropical rank of $M_F$ is $3$  and is  invariant under positive scaling of the columns of $M_F,$ which implies  $\dim\tconv(\lambda_1v_1,\ldots,\lambda_7v_7)\le2$ for any $\lambda_i>0$. 
If all 7 points are on the same ray we have that $\dim\tconv\pos (v_i)=1$ for each $i\in[7]$, since each ray is tropically convex. Hence, $\dim\tconv V=1$.
For the last case, suppose $V \subset \Gamma_F$ is such that $|\Supp V|<7$.
For each $i\in[7]$ let $V_i=\{\lambda_{i1}v_i,\ldots,\lambda_{ik_i}v_i\}\subset V$ and $\lambda_{\max_i}=\max\{\lambda_{i1},\ldots, \lambda_{ik_i}\}.$ Since each $V_i$ lies on a tropically convex ray, it follows that
$V_i \subseteq \tconv(0,\lambda_{\max_i}v_i) \subset \tconv(\lambda_{\max_1}v_1,\ldots, \lambda_{\max_7}v_7)$. Hence, $\tconv V\subset  \tconv(\lambda_{\max_1}v_1,\ldots, \lambda_{\max_7}v_7).$
The dimension of the tropical convex hull of any choice of the columns of $M_F$ is at most 2, hence $\dim\tconv V\le2.$

In order to prove that  $\dim\tconv\Gamma_F\leq2$ we use a similar argument to the one in the proof of Theorem~\ref{thm:degree general}.  Suppose that $\dim\tconv\Gamma_F=3.$ By Corollary~\ref{cor: dim of line segment}, 
$\tconv\Gamma_F$ contains a point $p$ with 4 distinct coordinates. Since $\Gamma_F$ is a fan, Corollary~\ref{Cor: properties} implies that $\tconv \Gamma_F$  contains the ray
$\pos(p)$, and we can choose $p$ to be the minimal nonnegative integer vector generating the ray. We may assume that $0=p_0<p_1<p_2<p_3.$ Let $a_1p, a_2p, a_3p,$ and $a_4p$ be four distinct points on $\pos(p)$ with $0<a_1<a_2<a_3<a_4.$ Let  $M_p$ be the
matrix with columns $a_ip$ for $i\in[4]$. Up to permutation of the rows, $M_p$ contains the $4\times4$ submatrix 

\begin{equation*}
   D =  \begin{pmatrix}
    0 & 0 & 0 & 0 \\
    a_1p_1 & a_2p_1 & a_3p_1 & a_4p_1 \\
    a_1p_2 & a_2p_2 & a_3p_2 & a_4p_2 \\
    a_1p_3 & a_2p_3 & a_3p_3 & a_4p_3
    \end{pmatrix}.
\end{equation*}
The tropical determinant of $D$ is $a_1p_3+a_2p_2+a_3p_1,$ and $D$ is tropically nonsingular. Hence, the tropical rank of $M_p$ is at least 4 and  $\dim \tconv(a_1p,\ldots,a_4p)\ge3.$
Each $a_ip \in \tconv\Gamma_F$ can be written as a tropical linear combination of a finite number of points on $\Gamma_F.$ Hence, $\tconv(a_1p,\ldots,a_4p)\subset\tconv W$ for a finite $W\subset\Gamma_F.$
This is a contradiction because $\dim \tconv W\leq2$ for all finite $W\subset\Gamma_F$. Thus $\dim \tconv\Gamma_F  = 2$.
\end{exmp}



\subsection*{Acknowledgements}
This material is based upon work supported by NSF-DMS grant \#1439786 while the authors were in residence at the Fall 2018 Nonlinear Algebra program at the Institute for Computational and Experimental Research in Mathematics in Providence, RI as well as their time spent at the Summer 2019 Collaborate@ICERM program. 
Cvetelina Hill was partially supported by NSF-DMS grant \#1600569. 
Sara Lamboglia was supported by  the LOEWE research unit Uniformized Structures in Arithmetic and Geometry.
Faye Pasley Simon was partially supported by NSF-DMS grant \#1620014.
The authors are particularly grateful to Josephine Yu for motivating this project,  helpful discussions, and a close reading. The authors would also like to thank  Marvin Hahn, Georg Loho, Diane Maclagan, and Ben Smith for useful feedback during the development of the project.

\bibliographystyle{alpha}
\bibliography{tropConvHulls}

\end{document}